\documentclass[11pt]{amsart} 
\usepackage{amsmath,amsthm}
\usepackage{xypic}
\usepackage{enumerate}

\usepackage{amscd}
\usepackage{amsfonts,amssymb,latexsym} 
\usepackage[all]{xy} 
\xyoption{arc}
\CompileMatrices

\newcommand{\udots}{\mathinner{\mskip1mu\raise1pt\vbox{\kern7pt\hbox{.}}
\mskip2mu\raise4pt\hbox{.}\mskip2mu\raise7pt\hbox{.}\mskip1mu}}

\newcommand{\SG}{{\mathcal{G}}}
\newcommand{\SH}{{\mathcal{H}}}
\newcommand{\SI}{{\mathcal{I}}}

\newcommand{\SM}{{\mathcal{M}}}
\newcommand{\SN}{{\mathcal{N}}}
\newcommand{\SO}{{\mathcal{O}}}
\newcommand{\SP}{{\mathcal{P}}}

\newcommand{\PP}{\mathbb{P}}
\newcommand{\ZZ}{\mathbb{Z}}

\newcommand{\CC}{\mathbb{C}}
\newcommand{\RR}{\mathbb{R}}

\newcommand{\Gr}{\operatorname{Gr}}
\newcommand{\Spec}{\operatorname{Spec}}

\newcommand{\Hom}{\operatorname{Hom}}

\newcommand{\id}{\operatorname{Id}}

\newcommand{\rk}{\operatorname{rk}}
\newcommand{\pdeg}{\operatorname{pardeg}}
\newcommand{\SParEnd}{\operatorname{SParEnd}}
\newcommand{\ParEnd}{\operatorname{ParEnd}}

\newcommand{\tr}{\operatorname{tr}}
\newcommand{\GL}{\operatorname{GL}}

\newcommand{\gitq}{/\!\!/}

\newcommand{\sP}{\operatorname{s-par}}

\newcommand{\op}{\operatorname}

\newtheorem{proposition}{Proposition}[section]
\newtheorem{theorem}[proposition]{Theorem}
\newtheorem{definition}[proposition]{Definition}

\newtheorem{lemma}[proposition]{Lemma}

\newtheorem{corollary}[proposition]{Corollary}

\numberwithin{equation}{section}

\title[Torelli parabolic Delienge-Hitchin]{Torelli theorem for the parabolic Deligne-Hitchin moduli space}
\author[D. Alfaya]{David Alfaya}
\author[T. G\'omez]{Tom\'as L. G\'omez}

\date{}

\address{Instituto de Ciencias Matem\'aticas (CSIC-UAM-UC3M-UCM),
Nicol\'as Cabrera 15, Campus Cantoblanco UAM, 28049 Madrid, Spain}

\email{david.alfaya@icmat.es}

\address{Instituto de Ciencias Matem\'aticas (CSIC-UAM-UC3M-UCM),
Nicol\'as Cabrera 15, Campus Cantoblanco UAM, 28049 Madrid, Spain}

\email{tomas.gomez@icmat.es}



\begin{document}

\begin{abstract}
We prove that, given the isomorphism class of the parabolic Deligne-Hitchin moduli space over a smooth projective curve, we can recover the isomorphism class of the curve and the parabolic points.
\end{abstract}

\maketitle

\section{Introduction}
\label{section:Intro}

Let $X$ be a compact connected Riemann surface of genus $g\ge 3$. The construction of the Deligne-Hitchin moduli space associated to $X$, $\SM_{\op{DH}}(X,r)$, is due to Deligne \cite{De}. In \cite{HitchinHiggs}, Hitchin built the twistor space for the hyper-K\"ahler structure of the moduli space $\SM_{\op{Higgs}}(X,r,\SO_X)$ and Simpson proved that this twistor space can be identified with the complex analytic space $\SM_{\op{DH}}(X,r)$ (see \cite[page 8]{Si2}).

In this paper we present a generalization of the Torelli theorem for the Deligne-Hitchin moduli space of a compact curve given in \cite{BGHL}. We will use the formalism of parabolic vector bundles in order to extend this theorem to punctured Riemann surfaces.

This paper is organized as follows. First, in Section \ref{section:ParVB}, we will deal with parabolic structures and provide some properties of the moduli space of parabolic vector bundles. Using them, in Section \ref{section:ParHiggsB}, we will be able to give an alternative proof to the Torelli theorem for the moduli space of parabolic Higgs bundles. This proof is slightly different from the one provided by \cite{GL} and adapts the techniques used in \cite{BGHL} to the parabolic case.

Parabolic $\lambda$-connections will be described in Section \ref{section:ParLambdaConn}, and we will consider a parabolic version of the Hodge moduli space for a punctured Riemann surface. Extending the techniques used in \cite{BGHL} and \cite{BGH2}, a Torelli theorem for the parabolic Hodge moduli space will be proven.

Finally, in Section \ref{section:ParDH}, we will use a parabolic version of the Riemann-Hilbert correspondence to construct the parabolic Deligne-Hitchin moduli space. The main result of this work is a Torelli theorem for this space (see Sections \ref{section:ParVB} and \ref{section:ParDH} for definitions and Theorem \ref{teo:TorelliParabolicDH}).

\begin{theorem}
Let $r=2$. Let $D$ be a set of $n\ge 1$ different points over a smooth complex projective curve $X$ of genus $g\ge 3$ and let $\alpha$ be a concentrated generic (in particular full flag) system of weights over $D$ such that for every $x\in D$,
\begin{equation}
\beta(x):=\sum_{i=1}^r \alpha_i(x) \in \ZZ
\end{equation}
and $\sum_{x\in D} \beta(x)$ is coprime with $r$. The isomorphism class of the complex analytic space $\SM_{\op{DH}}(X,r,\alpha)$ determines uniquely the isomorphism class of the unordered pair of punctured Riemann surfaces $\{(X,D),(\overline{X},D)\}$.
\end{theorem}

The rank two and coprimality conditions of the previous theorem are only necessary in order to apply the Torelli theorem in \cite{TorelliParabolic}. If the theorem of \cite{TorelliParabolic} were extended to higher rank, then Theorem \ref{teo:TorelliParabolicDH} would also hold for higher rank with the same proof given in this paper. Similarly, the Torelli theorem in \cite{TorelliParabolic} requires the parabolic weights to be chosen in a way that makes parabolic stability equivalent to the stability of the underlying vector bundle. If there existed a generalization of \cite{TorelliParabolic} for generic parabolic weights, the proof given in this article would hold not only for concentrated weights satisfying the coprimality condition, but for generic weights.

Nevertheless, the conditions of full flags, generic weights (see Section \ref{section:ParVB}) and $\beta(x)\in \ZZ$ are necessary in the current proof independently of \cite{TorelliParabolic}.

\noindent\textbf{Acknowledgments.} 
We thank Indranil Biswas for discussions. In particular, the proof of Proposition \ref{eq:forgetfulConnection} was simplified thanks to an idea of him.
This research was funded by 
MINECO (grant MTM2013-42135-P and ICMAT Severo Ochoa project
SEV-2015-0554) and the 7th European Union Framework Programme
(Marie Curie IRSES grant 612534 project MODULI). The first author was also supported by a predoctoral grant from Fundaci\'on La Caixa -- Severo Ochoa International Ph.D. Program.

\section{Parabolic Vector Bundles}
\label{section:ParVB}

Let $X$ be a smooth projective curve over $\CC$ of genus $g\ge 3$. Let $D$ be a finite set of $n\ge 1$ distinct points of $X$. We recall that a parabolic vector bundle over $X$ is a holomorphic vector bundle of rank $r$ together with a weighted flag on the fiber $E_x$ over each $x\in D$ called parabolic structure, i.e.
$$E_x=E_{x,0}\supsetneq E_{x,1} \supsetneq \cdots \supsetneq E_{x,l_x} = \{0\}$$
$$0\le \alpha_1(x)<\cdots <\alpha_{l_x}(x) <1$$
We denote by $\alpha=\{(\alpha_1(x),\ldots,\alpha_{l_x}(x))\}_{x\in D}$ the system of real weights corresponding to a fixed parabolic structure. We say that a parabolic structure is full flag if $l_x=r$ for all $x\in D$.

Equivalently \cite{SimpsonNonCompact}, we can describe the parabolic structure as a collection of decreasing left continuous filtrations, one filtration for each parabolic point. More precisely, for each $x\in D$, we have subsheaves $E_{\alpha}^x$ of $E$ on $X$, indexed by real 
$\alpha\ge 0$ such that
\begin{enumerate}[a)]
\item For every $\alpha\ge \beta$, $E_{\alpha}^x\subseteq E^x_{\beta}$
\item For every $\alpha>0$ there exist $\epsilon>0$ such that $E_{\alpha-\epsilon}^x=E_{\alpha}^x$
\item For every $\alpha$,
$E_{\alpha+1}^x=E_{\alpha}^x(-x)$
\item $E_{0}^x=E$
\end{enumerate}
Notice that conditions (a), (c) and (d) imply that the restriction of
$E_{\alpha}^x$ to $X\backslash \{x\}$ is isomorphic, for any $\alpha$, 
to $E|_{X\backslash \{x\}}$. Condition (c)
allows us to consider systems with arbitrary real weights, not necessarily between zero and one. Given
such a system, an equivalent parabolic vector bundle with weights
between zero and one can be obtained by tensoring with an appropriate
power of $\SO_X(x)$ for each $x\in D$. Using this definition,
a parabolic vector bundle is full flag if
$\dim \Gr(\{E_{\alpha}^x\}) \le 1$ for all $\alpha\ge 0$.

Let $\alpha$ be a fixed full flag parabolic structure and let $(E,E_\bullet)$ be a parabolic vector bundle over $X$ . The parabolic degree of $(E,E_\bullet)$ is defined as
$$
\pdeg(E,E_\bullet)=\deg(E)+\sum_{x\in D} \sum_{i=1}^{r} \alpha_i(x)
$$
and the parabolic slope is then
$$
\op{par}\mu(E,E_\bullet)=\frac{\pdeg(E,E_\bullet)}{\rk(E)}
$$

Let $(E,E_\bullet)$ be a parabolic vector bundle and $0\ne F\subseteq E$ a subbundle. Then $(E,E_\bullet)$ induces a parabolic structure on $F$ taking
$$
F_{\alpha}^x=F\cap E_{\alpha}^x
$$
In terms of flags on the fibers $F_x$ for $x\in D$, the induced parabolic structure on $F_x$ is given by removing all but the last duplicate subspaces of the filtration
$$
F_x = F\cap E_{x,0} \supseteq F\cap E_{x,1} \supseteq \cdots \supseteq F\cap E_{x,l_x}=\{0\}
$$
and taking the corresponding parabolic weights, i.e., for each $j$
\begin{eqnarray*}
\xymatrixcolsep{0pc}
\xymatrix{
F\cap E_x \ar@{=}[d] &  \supseteq & F\cap E_{x,i_{j-1}} \ar@{=}[d] &\supsetneq F\cap E_{x,i_{j-1}+1} = \ldots =& F\cap E_{x,i_j} \ar@{=}[d]  & \supsetneq F\cap E_{x,i_j+1}\\
F_x = F_{x,0} & \supseteq & F_{x,j-1} & \supsetneq& F_{x,j}&
}
\end{eqnarray*}
and we take $\alpha_{i_j}(x)$ as the weight for $F_{x,j}$. Let $I_F(x)=\{i_1,\ldots,i_k\}$. If the parabolic structure of $E$ is full flag, then the induced structure on $F$ is full flag too, as
$$
\dim \Gr(\{F_{\alpha}^x\})=\dim \Gr(\{F\cap E_{\alpha}^x\}) \le \dim \Gr(\{E_{\alpha}^x\})\le 1
$$
Then if $E$ is full flag, the parabolic degree of $F$ with the induced parabolic structure is
$$\pdeg(F,F_\bullet)=\deg(F) + \sum_{x\in D}\sum_{i\in I_F(x)} \alpha_i(x)$$
A parabolic bundle $(E,E_\bullet)$ is said to be stable (respectively semi-stable) if for all parabolic subbundles $F\subsetneq E$ with the induced parabolic structure we have
\begin{equation}
\label{eq:slopeconditon}
\op{par}\mu (F,F_\bullet) < \op{par}\mu(E,E_\bullet) \qquad (\text{respectively } \le )
\end{equation}

Let $\xi$ be a line bundle over $X$. Let $\SM(X,r,\alpha,\xi)$ be the moduli space of semi-stable parabolic vector bundles on $X$ of rank $r$ with weight system $\alpha$ together with an isomorphism $\bigwedge^r E \cong \xi$. We will omit the curve $X$ whenever it is clear. It is a projective scheme of dimension 
$$
\dim \SM(r,\alpha,\xi)=(g-1)(r^2-1)+\frac{n(r^2-r)}{2}
$$
Let $\SM^{\sP}(r,\alpha,\xi)$ be the open subset parameterizing the stable parabolic bundles. This open subvariety lies inside the smooth locus of $\SM(r,\alpha,\xi)$.

\begin{proposition}
Fix an integer $r\ge 2$. Then for a generic system of weights $\alpha = \{\alpha_1(x),\ldots,\alpha_r(x)\}_{x\in D}$, every full flag semi-stable parabolic vector bundle $(E,E_\bullet)$ over $X$ is stable.
\end{proposition}

\begin{proof}
Given a set $S$, and an integer $k$, let $\SP^k(S)$ denote the set of subsets of size $k$ of $S$. For each $0<r'<r$, each map $I:D\to \SP^{r'}(\{1,\ldots,r\})$ and each integer $-nr^2 \le m \le nr^2$, let
$$A_{I,m}=\{\alpha : r'\sum_{x\in D} \sum_{i=1}^r \alpha_i(x) - r \sum_{x\in D} \sum_{i\in I(x)} \alpha_i(x) =m\}$$
If we denote by $\SI_{r'}$ the set of possible maps $I:D\to \SP^{r'}(\{1,\ldots,r\})$, let
$$A=\bigcup_{r'=1}^{r-1} \bigcup_{I\in \SI_{r'}} \bigcup_{m=-nr^2}^{mr^2} A_{I,m}$$
Let us prove that for any system of weights $\alpha\not\in A$, every semi-stable full flag parabolic bundle is parabolically stable. Let us suppose that $(E,E_\bullet)$ is a strictly semi-stable parabolic vector bundle for the system of weights $\alpha$. Then there exist a subbundle $0\ne F \subsetneq E$ with the induced parabolic structure such that
\begin{multline*}
\frac{\deg(F) + \sum_{x\in D}\sum_{i\in I_F(x)} \alpha_i(x)}{\rk(F)} =\op{par}\mu(F,F_\bullet)\\
=\op{par}\mu(E,E_\bullet)=\frac{\deg(E) + \sum_{x\in D}\sum_{i=1}^r\alpha_i(x)}{r}
\end{multline*}
Then
$$\rk(F) \sum_{x\in D}\sum_{i=1}^r \alpha_i(x) - r \sum_{x\in D} \sum_{i\in I_F(x)} \alpha_i(x) = r \deg(F)-\rk(F) \deg(E)\in \ZZ$$
As $0\le \alpha_i(x)<1$ for every $x\in D$ and every $i=1,\ldots,r$
$$-nr^2\le \rk(F) \sum_{x\in D}\sum_{i=1}^r \alpha_i(x) - r \sum_{x\in D} \sum_{i\in I_F(x)} \alpha_i(x) \le nr^2$$
So $-nr^2 \le \deg(F)-\rk(F) \deg(E)\le nr^2$. Therefore
$$\alpha \in A_{I_F,r\deg(F)-\rk(F) \deg(E)} \subseteq A$$
\end{proof}

\begin{definition}
We say that a full flag system of weights $\alpha$ over $X$ is generic
if $\alpha\not\in A$, where $A$ is the closed subset of  $\RR^{nr}$
defined in the previous proposition. 
\end{definition}

\begin{corollary}
\label{lemma:parabolicvbsmooth}
Let $\xi$ be a line bundle over $X$. For a generic system of weights $\alpha$
$$\SM(r,\alpha,\xi)=\SM^{\sP}(r,\alpha,\xi)$$
\end{corollary}

We will now study the relation between stability of a parabolic vector bundle and the stability of its underlying vector bundle for a certain type of full flag systems of weights.

\begin{definition}
Fix a rank $r$. A full flag system of weights $\alpha=\{(\alpha_1(x),\ldots,\alpha_r(x))\}_{x\in D}$ is said to be concentrated if $\alpha_r(x)-\alpha_1(x)<\frac{4}{nr^2}$ for all $x\in D$.
\end{definition}

\begin{lemma}
\label{lemma:concentrated}
Let $\alpha$ be a concentrated system of weights. Let $I=\{1,\ldots, r\}$. Then for all $x\in D$ and for all $I'(x)\subsetneq I$, $I'(x)\ne \emptyset$, $|I'(x)|=r'$

$$\left | \sum_{x\in D} \left (r \sum_{i\in I'(x)} \alpha_i(x) - r' \sum_{i\in I} \alpha_i(x)\right) \right | <1$$
\end{lemma}

\begin{proof}
For each $I(x)\subsetneq I$, let $I^c(x)=I\backslash I(x)$. Then
$$\sum_{x\in D} \left ( r \sum_{i\in I'(x)} \alpha_i(x) - r' \sum_{i\in I} \alpha_i(x) \right) =\sum_{x\in D} \left ( (r-r') \sum_{i\in I'(x)} \alpha_i(x) - r' \sum_{i\in I'(x)} \alpha_i(x) \right)$$
As $\alpha_1(x) \le \alpha_i(x) \le \alpha_r(x)$, yields 
\begin{multline*}
\sum_{x\in D} \left ( (r-r') \sum_{i\in I'(x)} \alpha_i(x) - r' \sum_{i\in I'(x)} \alpha_i(x) \right) \\
\le \sum_{x\in D} \left ( (r-r') r' \alpha_r(x) - r' (r-r') \alpha_1(x) \right) < n(r-r')r' \frac{4}{nr^2}
\end{multline*}
Similarly
\begin{multline*}
\sum_{x\in D} \left ( (r-r') \sum_{i\in I'(x)} \alpha_i(x) - r' \sum_{i\in I'(x)} \alpha_i(x) \right)\\
\ge \sum_{x\in D} \left ( (r-r') r' \alpha_1(x) - r' (r-r') \alpha_r(x) \right) > -n(r-r')r' \frac{4}{nr^2}
\end{multline*}
Finally, the desired inequalities yield from arithmetic-geometric inequality, as for every $0\le r'<r$
$$(r-r)' r' \frac{4}{r^2} \le \left (\frac{ (r-r')+r'}{2} \right)^2 \frac{4}{r^2}=1$$
\end{proof}

\begin{proposition}
\label{prop:concentrated}
Let $\alpha$ be a full flag concentrated system of weights for rank $r$. Then for every parabolic vector bundle $(E,E_\bullet)$ over $X$ with $\op{gcd}(\deg(E),\rk(E))=1$ the following are equivalent

\begin{enumerate}
\item $E$ is semi-stable as a vector bundle
\item $E$ is stable as a vector bundle
\item $(E,E_\bullet)$ is semi-stable with respect to $\alpha$ as a parabolic vector bundle
\item $(E,E_\bullet)$ is stable with respect to $\alpha$ as a parabolic vector bundle
\end{enumerate}

\end{proposition}

\begin{proof}
Let $(E,E_\bullet)$ be a parabolic vector bundle over $X$.  We will prove $(1)\Rightarrow (2) \Rightarrow (4) \Rightarrow (3)\Rightarrow (1)$. By definition of (semi)-stability of a parabolic vector bundle $(4)\Rightarrow (3)$.

$(1)\Rightarrow (2)$ As a vector bundle, $E$ is semi-stable if and only if for every subbundle $\{0\}\ne F\ne E$ we have
$$\frac{\deg(F)}{\rk(F)} \le \frac{\deg(E)}{\rk(E)} \quad .$$
Or, equivalently, if and only if
$$\deg(F)\rk(E)-\deg(E)\rk(F) \le 0 \quad .$$
Let us suppose that $E$ is strictly semi-stable. Then there exist a subbundle $F$ with $0\ne F\ne E$ such that
$$\deg(F)\rk(E)-\deg(E)\rk(F)=0 \quad .$$
So $\deg(F)\rk(E)=\deg(E)\rk(F)$ and we have $\rk(E) | \deg(E)\rk(F)$. As $\op{gcd}(\deg(E),\rk(E))=1$ we must have $\rk(E) | \rk(F)$. Nevertheless, $F$ is a subbundle of $E$ with $\{0\}\ne F\ne E$, so $0<\rk(F)<\rk(E)$ and we arrive to a contradiction.

$(2)\Rightarrow (4)$ For every subbundle $F$ we consider the system of weights $\alpha_F$ induced by $\alpha$ on $F$. Then for every $x\in D$ there exist a subset $I_F(x)\subsetneq I$ with $|I_F(x)|=\rk(F)$ such that $\alpha_F=\{ \alpha_i : i\in I_F(x)\}_{x\in D}$. 
As a parabolic vector bundle, $(E,E_\bullet)$ is stable with respect to $\alpha$ if and only if for every subbundle $\{0\}\ne F\ne E$
$$\frac{\deg(F)+\sum_{x\in D}\sum_{i\in I_F(x)} \alpha_i(x)}{\rk(F)} < \frac{\deg(E)+\sum_{x\in D}\sum_{i\in I}\alpha_i(x)}{\rk(E)}$$
or equivalently, if and only if
$$\deg(F)\rk(E)-\deg(E)\rk(F)< \sum_{x\in D}\left (\rk(E) \sum_{i\in I_F(x)} \alpha_i(x) - \rk(F) \sum_{i\in I} \alpha_i(x) \right) \quad .$$
On the other hand if $E$ is stable then for every subbundle $0\ne F\ne E$ yields
$$\deg(F)\rk(E)-\deg(E)\rk(F)<0 \quad .$$
As $\deg(F)\rk(E)-\deg(E)\rk(F)\in \mathbb{Z}$ we have $\deg(F)\rk(E)-\deg(E)\rk(F)\le -1$. Lemma \ref{lemma:concentrated} implies
$$-1 < \sum_{x\in D}\left (\rk(E) \sum_{i\in I_F(x)} \alpha_i(x) - \rk(F) \sum_{i\in I} \alpha_i(x) \right) <1$$
so for every subbundle $\{0\}\ne F\ne E$
$$\deg(F)\rk(E)-\deg(E)\rk(F)\le -1 < \sum_{x\in D}\left (\rk(E) \sum_{i\in I_F(x)} \alpha_i(x) - \rk(F) \sum_{i\in I} \alpha_i(x)  \right) $$

$(3)\Rightarrow (1)$ If $(E,E_\bullet)$ is semi-stable with respect to $\alpha$ then for every subbundle $\{0\}\ne F\ne E$
$$
\deg(F)\rk(E)-\deg(E)\rk(F)\le \sum_{x\in D}\left (\rk(E) \sum_{i\in I_F(x)} \alpha_i(x) - \rk(F) \sum_{i\in I} \alpha_i(x)  \right) \quad .$$
As a consequence of Lemma \ref{lemma:concentrated},
$$\sum_{x\in D}\left (\rk(E) \sum_{i\in I_F(x)} \alpha_i(x) - \rk(F) \sum_{i\in I} \alpha_i(x)  \right) <1 \quad .$$
Therefore,
$$
\deg(F)\rk(E)-\deg(E)\rk(F)\le \sum_{x\in D}\left (\rk(E) \sum_{i\in I_F(x)} \alpha_i(x) - \rk(F) \sum_{i\in I} \alpha_i(x)  \right)<1 .$$
We have $\deg(F)\rk(E)-\deg(E)\rk(F)\in \mathbb{Z}$ and $\deg(F)\rk(E)-\deg(E)\rk(F)<1$, so for every subbundle $0\ne F\ne E$, $\deg(F)\rk(E)-\deg(E)\rk(F)\le 0$ and so, $E$ is semi-stable.
\end{proof}

The weights will be required to be concentrated in order to apply later on the Torelli theorem for the moduli space of rank 2 semi-stable parabolic vector bundles given in \cite{TorelliParabolic}. Nevertheless, if there existed a generalization of the Torelli theorem in \cite{TorelliParabolic} for generic parabolic weights, the proofs given in this paper would also hold for generic parabolic weights.

\begin{lemma}
\label{lemma:CotangentSections}
Let $\alpha$ be a generic (in particular full flag) system of weights. Then the holomorphic cotangent bundle
$$
T^*\SM(r,\alpha,\xi)\longrightarrow \SM(r,\alpha,\xi)
$$
does not admit any nonzero holomorphic section.
\end{lemma}

\begin{proof}
As a consequence of Corollary \ref{lemma:parabolicvbsmooth}, $\SM^{\sP}(r,\alpha,\xi)=\SM(r,\alpha,\xi)$. Then $\SM(r,\alpha,\xi)$ is smooth. As we are considering full flags on every point in $D$, by \cite[Theorem 6.1]{BY} $\SM(r,\alpha,\xi)$ is rational. Then it is a smooth rational projective variety, so it does not admit any nonzero holomorphic 1-form.
\end{proof}

\section{Parabolic Higgs bundles}
\label{section:ParHiggsB}

Let $(E,E_\bullet)$ be a parabolic vector bundle. A strongly parabolic endomorphism of $E$ is an endomorphism $\Phi:E\to E$ such that for every point $x\in D$,
$$\Phi(E_{x,i})\subseteq E_{x,i+1} \quad .
$$
Analogously, we say that an endomorphism $\Phi:E\to E$ is weakly parabolic if it satisfies
$$\Phi(E_{x,i})\subseteq E_{x,i} \quad .$$
Denote by $\SParEnd(E,E_\bullet)$ the sheaf of strongly parabolic endomorphisms of $(E,E_\bullet)$ and by $\ParEnd(E,E_\bullet)$ the sheaf of weakly parabolic endomorphisms.

A strongly parabolic Higgs bundle $(E,E_\bullet,\Phi)$ is a parabolic vector bundle $(E,E_\bullet)$ together with an homomorphism called Higgs field
$$\Phi : E\longrightarrow E\otimes K(D)
$$
such that for each $x\in D$ the homomorphism induced in the filtration over the fiber $E_x$ satisfies
$$
\Phi(E_{x,i}) \subseteq E_{x,i+1}\otimes K(D)|_x
$$
where $K$ is the canonical bundle over $X$ and $K(D)$ is the line bundle $K\otimes \SO_X(D)$. Similarly, a weakly parabolic Higgs bundle is a parabolic vector bundle $(E,E_\bullet)$ together with a Higgs field $\Phi:E\to E\otimes K(D)$ such that for each $x\in D$ the homomorphism induced in the filtration over the fiber $E_x$ satisfies
$$
\Phi(E_{x,i}) \subseteq E_{x,i}\otimes K(D)|_x
$$
We can view the Higgs field both as a holomorphic morphism
$\varphi:E\to E\otimes K(D)$  and as a meromorphic morphism
$\varphi:E\to E\otimes K$ that has at most simple poles over $D$. A
parabolic subbundle $(F,F_\bullet)\subset (E,E_\bullet)$ is said to be
$\Phi$-invariant if $\Phi(F)\subseteq F\otimes K(D)$. A parabolic
Higgs bundle (either strongly or weakly) is called (semi)-stable if
the stability slope condition \eqref{eq:slopeconditon} holds for every
$\Phi$-invariant subbundle $F\subsetneq E$, $F\ne \{0\}$. 
If we do not mention explicitly if a parabolic Higgs bundle is
strongly of weakly parabolic, it will be understood that it is
strongly parabolic.

We denote by $\SM_{\op{Higgs}}(r,\alpha,\xi)$ the moduli space of semi-stable (strongly) parabolic Higgs bundles of rank $r$, weight system $\alpha$ and $\tr\Phi=0$ together with an isomorphism $\bigwedge^r E \cong \xi$. It is an irreducible normal projective variety of dimension
$$\dim \SM_{\op{Higgs}}(r,\alpha,\xi)=2(g-1)(r^2-1)+n(r^2-r)$$
If $\alpha$ is a generic system of weights, then every semi-stable parabolic Higgs bundle is stable. The set of stable parabolic Higgs bundles lies within the smooth locus of $\SM_{\op{Higgs}}(r,\alpha,\xi)$. Therefore, if $\alpha$ is a generic system of weights $\SM_{\op{Higgs}}(r,\alpha,\xi)$ is smooth. Finally, let $\SM_{\op{Higgs}}^{w}(r,\alpha,d)$ be the moduli space of semi-stable weakly parabolic Higgs bundles of rank $r$ and weight system $\alpha$ such that the underlying vector bundle has degree $d$.

There is a natural embedding
\begin{equation}
\label{eq:embeddingHiggs}
i: \SM(r,\alpha,\xi) \hookrightarrow \SM_{\op{Higgs}}(r,\alpha,\xi)
\end{equation}
defined by $(E,E_\bullet) \mapsto (E,E_\bullet,0)$. Let $\SM_{\op{Higgs}}^{\sP}(r,\alpha,\xi)$ be the locus of parabolic Higgs bundles $(E,E_\bullet,\Phi)$ whose underlying parabolic vector bundle $(E,E_\bullet)$ is stable. It is an open dense subset of $\SM_{\op{Higgs}}(r,\alpha,\xi)$. Let
\begin{equation}
\label{eq:forgetfulHiggs}
\op{pr}_E: \SM_{\op{Higgs}}^{\sP}(r,\alpha,\xi) \longrightarrow \SM^{\sP}(r,\alpha,\xi)
\end{equation}
be the forgetful map defined by $(E,E_\bullet,\Phi)\to (E,E_\bullet)$. By deformation theory, the tangent space at $[(E,E_\bullet)]$, $T_{[(E,E_\bullet)]}\SM(r,\alpha,\xi)$ is isomorphic to $H^1(X,\ParEnd(E,E_\bullet))$. By the parabolic version of Serre duality,
$$H^1(X,\ParEnd(E,E_\bullet))^* \cong H^0(X,\SParEnd(E,E_\bullet) \otimes K(D))
$$
and hence, the Higgs field is an element of the cotangent bundle $T_{[(E,E_\bullet)]}^*\SM(r,\alpha,\xi)$ and one has a canonical isomorphism
\begin{equation}
\label{eq:HiggsCotangent}
\SM_{\op{Higgs}}^{\sP}(r,\alpha,\xi) \stackrel{\sim}{\longrightarrow} T^* \SM^{\sP}(r,\alpha,\xi)
\end{equation}
of varieties over $\SM^{\sP}(r,\alpha, \xi)$.

Corollary \ref{lemma:parabolicvbsmooth} implies that $\SM^{\sP}(r,\alpha,\xi) = \SM(r,\alpha,\xi)$ so we get an isomorphism 
\begin{equation}
\label{eq:Higgsdeformation}
\SM_{\op{Higgs}}^{\sP}(r,\alpha,\xi) \stackrel{\sim}{\longrightarrow} T^* \SM(r,\alpha,\xi)
\end{equation}
Let us recall the definition of the Hitchin map and the Hitchin space for weakly parabolic Higgs bundles.
Let $S=\mathbb{V}(K(D))$ be the total space of the line bundle K(D), let
$$p:S=\underline{\Spec}\op{Sym}^\bullet (K^{-1}\otimes \SO_X(D)^{-1}) \longrightarrow X
$$
be the projection, and $x\in H^0(S, p^*(K(D)))$ be the tautological section. The characteristic polynomial of a Higgs field
$$\det(x\cdot \op{id} - p^*\Phi)=x^r+ \widetilde{s_1}x^{r-1}+ \widetilde{s_2}x^{r-2}+\cdots+\widetilde{s_r}
$$
defines sections $s_i\in H^0(X,K^iD^i)$, such that $\widetilde{s_i}=p^*s_i$, where $K^iD^j$ denotes the tensor product of the $i$-th power of $K$ with the $j$-th power of the line bundle associated to $D$. We define the Hitchin space as
\begin{equation}
\SH = \bigoplus_{i=1}^r H^0(K^i D^i)
\end{equation}
The Hitchin map is defined as
\begin{equation}
\label{eq:HitchinMapNS}
H^{w}:\SM_{\op{Higgs}}^{w}(r,\alpha,d) \longrightarrow \SH
\end{equation}
sending each parabolic Higgs bundle $(E,E_\bullet,\Phi)$ to the characteristic polynomial of $\Phi$.

We can now restrict the Hitchin map to $\SM_{\op{Higgs}}(r,\alpha,\xi)$, where $\deg(\xi)=d$. We are assuming that $\Phi$ is strongly parabolic. Therefore the residue at each point of $D$ is nilpotent. This implies that the eigenvalues of $\Phi$ vanish at $D$, so for each $i>0$ the section $s_i$ belongs to the subspace $H^0(X,K^iD^{i-1})\subseteq H^0(X,K^iD^i)$. Moreover, in order to fix the determinant, we are asking $\Phi$ to be traceless, so $s_1=0$ and the image in the Hitchin space lies in

\begin{equation}
\SH_0 = \bigoplus_{i=2}^r H^0(K^i D^{i-1})
\end{equation}
Therefore, we obtain a map
\begin{equation}
\label{eq:HitchinMap}
H:\SM_{\op{Higgs}}(r,\alpha,\xi) \longrightarrow \SH_0
\end{equation}

\begin{lemma}
\label{lemma:HitchinProper}
The Hitchin map restricted to $\SM_{\op{Higgs}}(r,\alpha,\xi)$, \eqref{eq:HitchinMap} is projective.
\end{lemma}

\begin{proof}
By \cite[Corollary 5.12]{Y} the map \eqref{eq:HitchinMapNS} is projective. We clearly have an immersion
$$j:\SM_{\op{Higgs}}(r,\alpha,\xi) \longrightarrow \SM_{\op{Higgs}}^{w}(r,\alpha,d) \quad .$$
We can characterize $\SM_{\op{Higgs}}(r,\alpha,\xi)$ as the subset of $\SM_{\op{Higgs}}^{w}(r,\alpha,d)$ parameterizing semi-stable parabolic Higgs bundles $(E,E_\bullet,\Phi)$, such that $\det(E)\cong \xi$, $\tr \Phi=0$ and $\Phi$ is strongly parabolic. The condition of being strongly parabolic can be locally set imposing that for every $x\in D$ and every choice of local coordinates for the bundle around $x$ coherent with the given filtration, $\Phi$ has zeros in the diagonal entries. Therefore, all three conditions are closed and the map $j$ is a closed map. We now have the following commutative diagram
\begin{eqnarray*}
\xymatrixcolsep{3pc}
\xymatrix{
\SM_{\op{Higgs}}(r,\alpha,\xi) \ar[r]^-j \ar[d]_H & \SM_{\op{Higgs}}^{w}(r,\alpha,d) \ar[r]^-f  \ar[d]_{H^w}& \PP_{\SH}^n \ar[dl]\\
\SH_0 \ar[r] & \SH &
}
\end{eqnarray*}
As $H^{w}$ is projective, it factors into a closed immersion $f:\SM_{\op{Higgs}}^{w}(r,\alpha,d) \to \PP_{\SH}^n$ for some $n$, followed by the projection $\PP_{\SH}^n \to \SH$. Then $H^{w}\circ j$ factors into another closed immersion $f\circ j:\SM_{\op{Higgs}}(r,\alpha,\xi)\to \PP_{\SH}^n$ and the projection $\PP_{\SH}^n \to \SH$, so it is projective. As $H$ is just the restriction of the image of $j\circ H^{w}$ to $\SH_0$, $H$ must be projective.
\end{proof}

\begin{lemma}
\label{lemma:HitchinConnected}
The fibers of the Hitchin map \eqref{eq:HitchinMap} are connected.
\end{lemma}

\begin{proof}
By \cite[Lemma 3.1]{GL} and \cite[Lemma 3.2]{GL}, the fibers of \eqref{eq:HitchinMap} over a certain open dense subset $U$ of $\SH_0$ are isomorphic to a Prym variety, so each of those fibers are connected. Applying Stein factorization theorem \cite[Corollary 11.5]{Hartshorne77} to the projective morphism $H$ gives us an algebraic variety $\widetilde{\SH_0}$ and morphisms $\widetilde{H}$ and $g$ such that $\widetilde{H}$ has connected fibers, $g$ is a finite morphism and the following diagram commutes
\begin{eqnarray*}
\xymatrixcolsep{3pc}
\xymatrix{
\SM_{\op{Higgs}}(r,\alpha,\xi) \ar[r]^-{\widetilde{H}} \ar[dr]_H& {\widetilde{\SH_0}} \ar[d]_g \\
& \SH_0
}
\end{eqnarray*}
For every $p\in U$, $H^{-1}(p)$ is connected. The image of a connected set is connected, so $\widetilde{H}(H^{-1}(p))=g^{-1}(p)$ is connected. As $g$ is finite, $g^{-1}(p)$ must be a single point. Then $g$ is an isomorphism between $g^{-1}(U)$ and $U$, so $g$ is a birational map. Every finite morphism is projective, so by Zariski's Main Theorem \cite[Corollary 11.4]{Hartshorne77} a birational finite morphism to a normal variety is an isomorphism to an open set, so $g$ is an isomorphism to its image. Thus, every fiber of $H=\widetilde{H}\circ g$ is a fiber of $\widetilde{H}$ and must be connected.

\end{proof}

The multiplicative group $\CC^*$ acts on the moduli space $\SM_{\op{Higgs}}(r,\alpha,\xi)$ by
\begin{equation}
\label{eq:actionHiggs}
t\cdot (E,E_\bullet,\Phi)=(E,E_\bullet,t\Phi)
\end{equation}
The Hitchin map $H$ induces an associated $\CC^*$ action in $\SH$ given by

\begin{equation}
\label{eq:actionHitchin}
t\cdot (v_2,\ldots,v_i,\ldots,v_r) = (t^2 v_2,\ldots,t^iv_i,\ldots,t^r v_r)
\end{equation}
Where $v_i\in H^0(X,K^i D^{i-1})$ for $i\in {2,\ldots,r}$.

\begin{lemma} 
\label{lemma:HiggsSections}
Let $\alpha$ be a generic system of weights. Then the holomorphic tangent bundle
$$T\SM(r,\alpha,\xi)\longrightarrow \SM(r,\alpha,\xi)
$$
does not admit any nonzero holomorphic section.
\end{lemma}

\begin{proof}
A holomorphic section $s$ of $T\SM(r,\alpha,\xi)$ provides by contraction a holomorphic function
$$s^\sharp:T^*\SM(r,\alpha,\xi) \longrightarrow \CC$$
on the total space of the cotangent bundle which is linear on the fibers. Under the isomorphism in \eqref{eq:Higgsdeformation}, it corresponds to a holomorphic function
$$f:\SM_{\op{Higgs}}^{\sP}(r,\alpha,\xi)\longrightarrow \CC$$
Taking $\SG = \op{SL}(\SO_X^{\oplus (r-1)}\oplus \xi)$ in \cite[Lemma II.6]{Faltings} and \cite[V.(iii), page 561]{Faltings} we obtain that the codimension of $\SM_{\op{Higgs}}^{\sP}(r,\alpha,\xi)$ in $\SM_{\op{Higgs}}(r,\alpha,\xi)$ is grater than two.

As $\SM_{\op{Higgs}}(r,\alpha,\xi)$ is smooth, by Hartog's theorem, the function $f$ extends to a holomorphic function

$$\widetilde{f}: \SM_{\op{Higgs}}(r,\alpha,\xi) \longrightarrow \CC$$

Since $f$ is linear on the fibers we know that $\widetilde{f}$ must be homogeneous of degree 1 for the action \eqref{eq:actionHiggs} of $\CC^*$. On the moduli space $\SM_{\op{Higgs}}(r,\alpha,\xi)$ the Hitchin map \eqref{eq:HitchinMap} is projective by Lemma \ref{lemma:HitchinProper}, so it is proper, and its fibers are connected by Lemma \ref{lemma:HitchinConnected}. Therefore, the function $\widetilde{f}$ is constant on the fibers of the Hitchin map and $\widetilde{f}$ factors through a holomorphic function on the Hitchin space, which must still be homogeneous of degree 1.

On the other hand, there is no nonzero holomorphic homogeneous function of degree 1 on $\SH$, because all the exponents of $t$ in \eqref{eq:actionHitchin} are at least two. Therefore, $\widetilde{f}=0$ and we get $f=0$, $s^\sharp=0$ and, finally, $s=0$.
\end{proof}

\begin{corollary}
\label{cor:TangentHiggsSections}
Let $\alpha$ be a generic (in particular full flag) system of weights. The restriction of the holomorphic tangent bundle
$$T\SM_{\op{Higgs}}(r,\alpha,\xi) \longrightarrow \SM_{\op{Higgs}}(r,\alpha,\xi)
$$
to $i(\SM(r,\alpha,\xi))\subseteq \SM_{\op{Higgs}}(r,\alpha,\xi)$ does not admit any nonzero holomorphic section.
\end{corollary}

\begin{proof}
Using Lemma \ref{lemma:HiggsSections}, it suffices to show that the normal bundle of the embedding
$$i:\SM(r,\alpha,\xi) \longrightarrow \SM_{\op{Higgs}}^{\sP}(r,\alpha,\xi)$$
does not admit any nonzero holomorphic sections. The isomorphism in \eqref{eq:Higgsdeformation} allows us to identify this normal bundle with $T^*\SM(r,\alpha,\xi)$, so the lemma follows from Lemma \ref{lemma:CotangentSections}.
\end{proof}

We can adapt Simpson's result \cite[Lemma 11.9]{Simpson} to the parabolic situation and we obtain the following 

\begin{lemma}
\label{lemma:Simpson}
Let $(E,E_\bullet,\Phi)$ be a parabolic Higgs bundle in the nilpotent cone, with $\Phi\ne 0$. Assume that $(E,E_\bullet,\Phi)$ is a fixed point of the action \eqref{eq:actionHiggs}. Then there is another Higgs bundle $(E',E'_\bullet,\Phi')$ in the nilpotent cone, not isomorphic to $(E,E_\bullet,\Phi)$ such that $\lim_{t\to \infty} (E',E'_\bullet,t\Phi') = (E,E_\bullet,\Phi)$
\end{lemma}

The previous results combine in

\begin{proposition}
\label{prop:Simpson}
Let $\alpha$ be a full flag generic system of weights. Let $Z$ be an irreducible component of the fixed point locus of the action \eqref{eq:actionHiggs} in $\SM_{\op{Higgs}}(r,\alpha,\xi)$. Then

$$\dim Z \le (r^2-1)(g-1)+ \frac{n(r^2-r)}{2}
$$
with equality only for $Z=i(\SM(r,\alpha,\xi))$.
\end{proposition}

\begin{proof}
The $\CC^*$ action \eqref{eq:actionHiggs} on $\SM_{\op{Higgs}}(r,\alpha,\xi)$ and the $\CC^*$ action \eqref{eq:actionHitchin} on $\SH_0$ are intertwined by the Hitchin map $H$. Clearly the only fixed point in $\SH_0$ for this action is $0$, so $Z\subseteq H^{-1}(0)$.

By \cite[Theorem 3.14]{GGM07}, the fiber $H^{-1}(0)$ is a Lagrangian subscheme of $\SM_{\op{Higgs}}(r,\alpha,\xi)$ so $\dim H^{-1}(0)= \frac{1}{2}\dim \SM_{\op{Higgs}}(r,\alpha,\xi) =\dim \SM(r,\alpha,\xi)$ and so

$$\dim Z \le \dim \SM(r,\alpha,\xi)=(r^2-1)(g-1)+ \frac{n(r^2-r)}{2}
$$
and equality holds if $Z$ is an irreducible component of $H^{-1}(0)$.

Recall that $i:\SM(r,\alpha,\xi) \to H^{-1}(0)$ takes $E\to (E,0)$. Since a non-trivial $\CC^*$-action produces a non-trivial vector field, from Lemma \ref{lemma:HiggsSections} we know that $\SM(r,\alpha,\xi)$ does not admit any non-trivial $\CC^*$-action. As $\dim \SM(r,\alpha,\xi) =\dim H^{-1}(0)$, $i(\SM(r,\alpha,\xi))$ is an irreducible component of $H^{-1}(0)$ of the maximum allowed dimension, so it remains to check that there is no other irreducible component where there is no $\CC^*$-action.
The rest of the components have a nonzero Higgs field, so the $\CC^*$- action \eqref{eq:actionHiggs} $(E,\Phi)\mapsto (E,t\Phi)$ is non-trivial due to Lemma \ref{lemma:Simpson}.
\end{proof}

Using the previous Proposition we can obtain a proof of the Torelli theorem for the moduli space of parabolic Higgs bundles on a curve

\begin{corollary}
\label{cor:torelliHiggs}
Let $r=2$. Let $X$ be a smooth complex projective curve of genus $g\ge 3$. Let $D$ be a finite set of $n\ge 1$ distinct points over $X$ and let $\alpha$ be a concentrated generic (in particular full flag) system of weights over $D$. Let $\xi$ be a line bundle over $X$  such that $\deg(\xi)$ is coprime with $r$. The isomorphism class of the complex analytic space $\SM_{\op{Higgs}}(r,\alpha,\xi)$ determines uniquely the isomorphism class of the punctured Riemann surface $(X,D)$, meaning that if $\SM_{\op{Higgs}}(X,r,\alpha,\xi)$ is biholomorphic to $\SM_{\op{Higgs}}(X',r,\alpha',\xi')$ for another punctured connected Riemann surface $(X',D')$ of the same genus $g$ and a generic concentrated system of weights $\alpha'$ over $D'$, then $(X,D)\cong (X',D')$.
\end{corollary}

\begin{proof}
Let $Z\subset \SM_{\op{Higgs}}(r,\alpha,\xi)$ be a closed analytic subset with the following two properties:

\begin{enumerate}
\item $Z$ is irreducible, smooth, and has complex dimension $(r^2-1)(g-1)+ \frac{n(r^2-r)}{2}$.
\item The restriction of the holomorphic tangent bundle $T\SM_{\op{Higgs}}(r,\alpha,\xi)$ to $Z\subset \SM_{\op{Higgs}}(r,\alpha,\xi)$ has no nonzero holomorphic sections.
\end{enumerate}

By Corollary \ref{cor:TangentHiggsSections}, the image
$i(\SM(r,\alpha,\xi))$ of the embedding $i$ in
\eqref{eq:embeddingHiggs} has these properties. We will prove that
this is the only possible choice for $Z$.

Every $\CC^*$ action on $\SM_{\op{Higgs}}(r,\alpha,\xi)$ defines a
holomorphic vector field on $\SM_{\op{Higgs}}(r,\alpha,\xi)$. The
second assumption on $Z$ says that any holomorphic vector field on
$\SM_{\op{Higgs}}(r,\alpha,\xi)$ vanishes on $Z$. Therefore, the
stabilizer of each point $Z\subset \SM_{\op{Higgs}}(r,\alpha,\xi)$ has
nontrivial tangent space at $1\in \CC^*$, and hence the stabilizer
must be the full group $\CC^*$.

Then $Z$ belongs to the fixed point locus of the action
\eqref{eq:actionHiggs} in $\SM_{\op{Higgs}}(r,\alpha,\xi)$. 
Due to Proposition \ref{prop:Simpson}, and property (1), 
$Z=i(\SM(r,\alpha,\xi))$. In particular, we have $Z\cong \SM(r,\alpha,\xi)$.

Then, the isomorphism class of $\SM_{\op{Higgs}}(r,\alpha,\xi)$
determines the isomorphism class of $\SM(r,\alpha,\xi)$. Due to
\cite[Theorem 3.2]{TorelliParabolic}, the latter determines the
isomorphism class of the punctured Riemann surface $(X,D)$.
\end{proof}

The rank two condition of the previous corollary is only necessary in order to apply the Torelli theorem in \cite{TorelliParabolic}. If the Theorem of \cite{TorelliParabolic} were extended to Higher rank, then Corollary \ref{cor:torelliHiggs} would also hold for higher rank with the same proof given above.

\section{The parabolic $\lambda$-connections}
\label{section:ParLambdaConn}

Let $\xi$ be a line bundle over $X$ and let $\alpha$ be a fixed full flag system of weights over $D$. Let us suppose that $\deg(\xi)=- \sum_{x\in D}\sum_{i=1}^r \alpha_i(x)$. Fixing a line bundle and a system of weights $\alpha$ over $X$ allows us to describe canonically a parabolic line bundle over $X$, $(\xi,\xi_{\beta}^x)$, taking the underlying vector bundle as $\xi$ and defining trivial filtrations over each $x\in D$ with parabolic weight
$$\beta(x):=\beta_1(x)=\sum_{i=1}^r \alpha_i(x)$$
As $\xi$ has rank one, any parabolic structure on $\xi$ consists of trivial filtrations. It is possible that for some $x\in D$, $\beta(x)\ge 1$. Taking into account the definition for the parabolic structure in terms of left continuous filtrations given by Simpson \cite{SimpsonNonCompact}, a parabolic line bundle $\xi$ with jumps at weights $\beta(x)$ for each $x\in D$ such that $\xi_{\beta(x)}^x=\xi(-x)$ is the same as a trivial filtration for the bundle
\begin{equation}
\xi \left (\sum_{x\in D} \lfloor \beta(x) \rfloor x \right)
\end{equation}
with parabolic weights $\{\beta(x)-\lfloor \beta(x) \rfloor\}_{x\in D}$.

Thus, the value of the jump $\beta(x)$ defines the parabolic structure on $\xi$ completely. By construction, we get that
$$\pdeg(\xi)=\deg(\xi) + \sum_{x\in D}\beta(x)=\deg(\xi)+\sum_{x\in D}\sum_{i=1}^r \alpha_i(x)=0$$
The line bundle $\xi$ can be given the structure of a parabolic Higgs bundle canonically taking a zero Higgs field. In fact, as the rank of $\xi$ is one, every traceless Higgs field over $\xi$ must be zero, so $\SM_{\op{Higgs}}(1,\beta,\xi)$ consists exactly of the point $(\xi,0)$.

Let $(E,E_\bullet,\Phi)$ be a traceless strongly parabolic Higgs
bundle with parabolic system of weights $\alpha$ such that
$\det(E)=\xi$. Taking the determinant of $(E,E_\bullet,\Phi)$,
we obtain a rank one parabolic Higgs bundle
$(\det(E),\det(E)_\bullet,\tr\Phi)$.
As $\tr(\Phi)=0$, the induced morphism is the zero
morphism.

Thus, taking the determinant, every parabolic Higgs bundles $[(E,E_\bullet,\Phi)]\in \SM_{\op{Higgs}}(r,\alpha,\xi)$ induces the same parabolic Higgs bundle $(\xi,\xi_\bullet,0)$.

Using the Simpson correspondence \cite{SimpsonNonCompact} between parabolic Higgs bundles of parabolic degree 0 and parabolic connections of parabolic degree 0, the parabolic Higgs bundle $(\xi,\xi_\bullet,0)$ corresponds to a parabolic connection $(\xi',\xi'_\bullet,\nabla_{\xi,\beta})$ with the same parabolic weights $\beta$, such that $\op{Res}(\nabla_{\xi,\beta},x)=\beta(x) \id$ for every $x\in D$.

Let $(E',E'_\bullet,\nabla)$ be the parabolic connection corresponding
to the Higgs bundle $(E,E_\bullet,\Phi)$ under the Simpson
correspondence. Taking the determinant of $(E',E'_\bullet,\nabla)$ we
obtain a rank one parabolic connection
$$
(\det(E),\det(E)_\bullet,\widetilde{\nabla}=\tr(\nabla)) \quad
.
$$
As the Simpson correspondence is an equivalence of categories
compatible with determinants \cite[Theorem 2]{SimpsonNonCompact}, the
determinant of $(E',E'_\bullet,\nabla)$ must be the image of
the determinant of $(E,E_\bullet,\Phi)$. Therefore, the
morphism $\widetilde{\nabla}$ must coincide with
$\nabla_{\xi,\beta}$. This leads up to the following definition of
parabolic $\lambda$-connection for the group $\op{SL}(r,\CC)$.

\begin{definition}
\label{def:parabolicLambdaConnection}
For a fixed line bundle $\xi$, a system of weights $\alpha$ and a given $\lambda \in \CC$ a parabolic $\lambda$-connection on $X$ (for the group $\op{SL}(r,\CC)$) is a triple $(E,E_\bullet,\nabla)$ where

\begin{enumerate}
\item $(E,E_\bullet)\longrightarrow X$ is a parabolic vector bundle of rank $r$ and weight system $\alpha$ together with an isomorphism $\bigwedge^r E\cong \xi$
\item $\nabla : E\to E\otimes K(D)$ is a $\CC$-linear homomorphism of sheaves over the underlying vector space of $E$ satisfying the following conditions
\begin{enumerate}
\item If $f$ is a locally defined holomorphic function on $X$ and $s$ is a locally defined holomorphic section of $E$ then
$$\nabla (fs) = f\cdot \nabla(s) + \lambda \cdot s\otimes df$$
\item For each $x\in D$ the homomorphism induced in the filtration over the fiber $E_x$ satisfies
$$\nabla(E_{x,i}) \subseteq E_{x,i} \otimes K(D)|_x$$
\item For every $x\in D$ and every $i=1,\ldots, r$ the action of $\op{Res}(\nabla,x)$ on $E_{x,i}/E_{x,i-1}$ is the multiplication by $\lambda \alpha_i(x)$. Since $\op{Res}(\nabla,x)$ preserves the filtration, it acts on each quotient.
\item The operator $\tr(\nabla):\bigwedge^r E \longrightarrow \left(\bigwedge^r E \right) \otimes K(D)$ induced by $\nabla$ coincides with $\lambda \cdot \nabla_{\xi,\beta}$.
\end{enumerate}
\end{enumerate}
\end{definition}

We say that a parabolic $\lambda$-connection $(E,E_\bullet,\nabla)$ on $X$ (for the group $\op{SL}(r,\CC)$) is stable (respectively semi-stable) if and only if for every parabolic subbundle $F\subsetneq E$ with the induced parabolic structure such that $\nabla(F)\subseteq F$
$$\op{par}\mu (F) < \op{par}\mu(E) \qquad (\text{respectively }\le )$$
We denote by $\SM_{\op{Hod}}(r,\alpha,\xi)$ the moduli space of semi-stable quadruples of the form $(\lambda,E,E_\bullet,\nabla)$, where $\lambda$ is a complex number and $(E,E_\bullet,\nabla)$ is a parabolic $\lambda$-connection. In the non-parabolic case, the Hodge moduli space was described by Deligne \cite{De} and it was constructed as an instance of a moduli space of $\Lambda$-modules by Simpson \cite{Si2}. For the construction of this moduli space in the parabolic scenario, see \cite{Al16}. The moduli space $\SM_{\op{Hod}}(r,\alpha,\xi)$ is a complex algebraic variety of dimension $1+2(g-1)(r^2-1)+n(r^2-r)$. It is equipped with a surjective algebraic morphism
\begin{equation}
\op{pr}_\lambda:\SM_{\op{Hod}}(r,\alpha,\xi) \longrightarrow \CC
\end{equation}
defined by $\op{pr}_\lambda(\lambda,E,E_\bullet,\nabla)=\lambda$.

Given a parabolic vector bundle $(E,E_\bullet)$, taking $\lambda=0$, a parabolic $0$-connection over $(E,E_\bullet)$ is a homomorphism of $\SO_X$-modules $\nabla:E\to E\otimes K(D)$ that preserves the filtration and such that for every $x\in D$, $\op{Res}(\nabla,x)$ acts as the zero morphism on $E_{x,i}/E_{x,i+1}$. Then, for every $x\in D$, $\nabla(E_{x,i})\subseteq E_{x,i+1}\otimes K(D)|_x$. Moreover, the induced morphism $\bigwedge^r E\to \bigwedge^r E\otimes K(D)$ is zero. As this morphism coincides locally with the trace of $\nabla$, $\nabla$ is a traceless morphism $\nabla:E\to E\otimes K(D)$. 

Thus, a $0$-connection is a traceless strongly parabolic Higgs bundle, so
$$\SM_{\op{Higgs}}(r,\alpha,\xi)=pr_\lambda^{-1}(0)\subset \SM_{\op{Hod}}(r,\alpha,\xi)$$

In particular, the embedding \eqref{eq:embeddingHiggs} of $\SM(r,\alpha,\xi)$ into $\SM_{\op{Higgs}}(r,\alpha,\xi)$ also gives an embedding of $\SM(r,\alpha,\xi)$ into $\SM_{\op{Hod}}(r,\alpha,\xi)$

\begin{equation}
\label{eq:embeddingHodge}
i:\SM(r,\alpha,\xi) \hookrightarrow \SM_{\op{Hod}}(r,\alpha,\xi)
\end{equation}
and the group $\CC^*$ acts on $\SM_{\op{Hod}}(r,\alpha,\xi)$ extending the $\CC^*$ action on $\SM_{\op{Higgs}}(r,\alpha,\xi)$ introduced in formula \eqref{eq:actionHiggs} by

\begin{equation}
\label{eq:actionHodge}
t\cdot (\lambda, E,E_\bullet,\nabla) = (t\cdot \lambda, E,E_\bullet, t\cdot \nabla)
\end{equation}

\begin{proposition}
\label{prop:SimpsonHodge}
Let $\alpha$ be a full flag generic system of weights. Let $Z$ be an irreducible component of the fixed point locus of the action \eqref{eq:actionHodge} in $\SM_{\op{Hod}}(r,\alpha,\xi)$. Then $\dim Z \le (r^2-1)(g-1)+\frac{n(r^2-r)}{2}$, with equality only for $Z=i(\SM(r,\alpha,\xi))$
\end{proposition}

\begin{proof}
A point $(\lambda,E,E_\bullet,\nabla)\in \SM_{\op{Hod}}(r,\alpha,\xi)$
can only be a fixed point of the action 
if $\lambda=0$. Then $Z\subseteq \SM_{\op{Higgs}}(r,\alpha,\xi)$. The
result follows from Proposition \ref{prop:Simpson}.
\end{proof}

A $\lambda$-connection in $\SM_{\op{Hod}}(r,\alpha,\xi)$ for $\lambda=1$ is a holomorphic flat connection on a parabolic vector bundle in the usual way \cite[\S 3]{Biswas}, that is, a logarithmic connection singular over $D$ such that the residue at every $x\in D$ restricted to $E_{x,i}/E_{x,i-1}$ is just the multiplication by $\alpha_i(x)$, so

$$ \SM_{\op{conn}}(r,\alpha,\xi) :=\op{pr}_\lambda^{-1}(1) \subset \SM_{\op{Hod}}(r,\alpha,\xi)
$$
is the moduli space of parabolic $\op{SL}(r,\CC)$-connections $(E,E_\bullet,\nabla)$ with weight system $\alpha$ and an isomorphism $\det(E) \cong \xi$. We denote by

$$\SM_{\op{conn}}^{\sP}(r,\alpha,\xi) \subset \SM_{\op{conn}}(r,\alpha,\xi) \quad \text{and} \quad \SM_{\op{Hod}}^{\sP}(r,\alpha,\xi)\subset \SM_{\op{Hod}}(r,\alpha,\xi)
$$
the Zariski open subvarieties where the underlying parabolic vector bundle is stable.

\begin{proposition}
\label{prop:connSections}
The forgetful map

\begin{equation}
\label{eq:forgetfulConnection}
\op{pr}_E : \SM_{\op{conn}}^{\sP}(r,\alpha,\xi) \longrightarrow \SM^{\sP}(r,\alpha,\xi)
\end{equation}
defined by $\op{pr}_E(E,\nabla)=E$ admits no holomorphic section.
\end{proposition}

\begin{proof}
Let us suppose that there exist a holomorphic section
$$s:\SM^{\sP}(r,\alpha\xi)\longrightarrow \SM_{\op{conn}}^{\sP}(r,\alpha,\xi)$$
By the Riemann-Hilbert correspondence, the map that sends each parabolic connection to its monodromy is a holomorphic map from $\SM_{\op{conn}}^{\sP}(r,\alpha,\xi)$ to an algebraically closed subset of the moduli space of irreducible representations of the fundamental group of $X\backslash D$ in $\GL(r,\CC)$. In particular, given a base point $x_0\in X\backslash D$, we obtain a holomorphic map (see Section \ref{section:ParDH} for the details)
$$\op{RH}^{-1}:\SM_{\op{conn}}^{\sP}(r,\alpha,\xi)\longrightarrow \Hom(\pi_1(X\backslash D,x_0),\GL(r,\CC))\gitq \GL(r,\CC)$$
The latter is an affine variety because the GIT quotient of an affine variety is affine and $\Hom(\pi_1(X\backslash D,x_0),\GL(r,\CC))$ is a closed subset of $\mathbb{C}^{(r^2+1)(2g+n)}$. On the other hand, as we assumed generic parabolic weights, $\SM^{\sP}(r,\alpha,\xi)=\SM(r,\alpha,\xi)$ is compact as an analytic variety. Therefore,
$$\op{RH}^{-1}\circ s:\SM^{\sP}(r,\alpha,\xi) \longrightarrow \Hom(\pi_1(X\backslash D,x_0),\GL(r,\CC))\gitq \GL(r,\CC)$$
is a holomorphic morphism from a compact variety to an affine variety and it must be constant. As the parabolic data has been fixed, by \cite{Ka76} for every representation of the fundamental group in the image of $\op{RH}^{-1}$ there exist a unique logarithmic connection on $X$ with the prescribed monodromy and residual data. Moreover, the parabolic structure is assumed to be full flag, so there is a single parabolic structure on the underlying vector bundle compatible with the logarithmic connection.Therefore, the morphism $\op{RH}^{-1}:\SM_{\op{conn}}^{\sP}(r,\alpha,\xi)\longrightarrow \Hom(\pi_1(X\backslash D,x_0),\GL(r,\CC))\gitq \GL(r,\CC)$ is injective, so $s$ must be constant and we conclude that it is not a section of $\op{pr}_E$.

\end{proof}

The forgetful maps \eqref{eq:forgetfulConnection} and \eqref{eq:forgetfulHiggs} can be both seen as restrictions to $\op{pr}_\lambda^{-1}(0)$ and $\op{pr}_\lambda^{-1}(1)$ respectively of a map

\begin{equation}
\label{eq:forgetfulHodge}
\op{pr}_E: \SM_{\op{Hod}}^{\sP}(r,\alpha,\xi) \longrightarrow \SM^{\sP}(r,\alpha,\xi)
\end{equation}
defined by $\op{pr}_E(\lambda,E,E_\bullet,\nabla)=(E,E_\bullet)$.

\begin{corollary}
\label{cor:HodgeSections}
Let $\alpha$ be a generic (in particular full flag) system of weights. The only holomorphic map
$$s:\SM(r,\alpha,\xi)\longrightarrow \SM_{\op{Hod}}^{\sP}(r,\alpha,\xi)$$
with $\op{pr}_E \circ s =\id$ is the restriction of the embedding $i$ defined in \eqref{eq:embeddingHodge}
$$i:\SM(r,\alpha,\xi)\hookrightarrow \SM_{\op{Hod}}^{\sP}(r,\alpha,\xi)$$
\end{corollary}

\begin{proof}
The composition
$$\SM(r,\alpha,\xi) \stackrel{s}{\longrightarrow} \SM^{\sP}_{\op{Hod}}(r,\alpha,\xi) \stackrel{\op{pr}_\lambda}{\longrightarrow} \CC
$$
is a holomorphic function on $\SM(r,\alpha,\xi)$. Since the later is compact, it is a constant function. Up to the $\CC^*$ action in \eqref{eq:actionHodge}, we may assume that this constant is either $0$ or $1$.

If this constant were 1, then $s$ would factor through $\op{pr}_\lambda^{-1}(1)=\SM_{\op{conn}}^{\sP}(r,\alpha,\xi)$, which would contradict Proposition \ref{prop:connSections}.
Hence this constant is 0, and $s$ factors through $\op{pr}_\lambda^{-1}(0)=\SM_{\op{Higgs}}^{\sP}(r,\alpha,\xi)$. Thus, under isomorphism \eqref{eq:Higgsdeformation} $s$ corresponds to a holomorphic global section of $T^*\SM(r,\alpha,\xi)$. But due to Lemma \ref{lemma:CotangentSections}, $s$ is zero, so it must be the restriction of the canonical embedding $i$ in \eqref{eq:embeddingHodge}.
\end{proof}

\begin{corollary}
\label{cor:TangentHodgeSections}
Let $\alpha$ be a generic (in particular full flag) system of weights. Then every semi-stable $\lambda$-connection is stable. As the set of stable $\lambda$-connections lie in the smooth locus of $\SM_{\op{Hod}}(r,\alpha,\xi)$, the later is smooth. The restriction of the holomorphic tangent bundle
$$T\SM_{\op{Hod}}(r,\alpha,\xi) \longrightarrow \SM_{\op{Hod}}(r,\alpha,\xi)$$
to $i(\SM(r,\alpha,\xi))\subset \SM_{\op{Hod}}(r,\alpha,\xi)$ does not admit any nonzero holomorphic section.
\end{corollary}

\begin{proof}
Let $\SN$ be the holomorphic normal bundle of the restricted embedding
$$i:\SM(r,\alpha,\xi) \hookrightarrow \SM_{\op{Hod}}(r,\alpha,\xi)$$
Due to Lemma \ref{lemma:HiggsSections}, it suffices to show that this vector bundle $\SN$ over $\SM(r,\alpha,\xi)$ has no nonzero holomorphic sections.
One has a canonical isomorphism
\begin{equation}
\label{eq:HodgeNormal}
\SM_{\op{Hod}}(r,\alpha,\xi) \stackrel{\sim}{\longrightarrow} \SN
\end{equation}
of varieties over $\SM^s(r,\alpha,\xi)$, defined by sending any $(\lambda,E,E_\bullet,\nabla)$ to the derivative at $t=0$ of the map $\CC\longrightarrow \SM_{\op{Hod}}(r,\alpha,\xi)$ given by

$$t\longmapsto (t\cdot \lambda, E,E_\bullet, t\cdot \nabla)$$

Using this morphism, from Corollary \ref{cor:HodgeSections} we conclude that $\SN$ does not have any nonzero holomorphic sections.
\end{proof}

\begin{corollary}
\label{cor:torelliHodge}
Let $r=2$.  Let $X$ be a smooth complex projective curve of genus $g\ge 3$. Let $D$ be a finite set of $n\ge 1$ distinct points over $X$ and let $\alpha$ be a concentrated generic (in particular full flag) system of weights over $D$. Let $\xi$ be a line bundle over $X$  such that $\deg(\xi)=-\sum_{x\in D}\sum_{i=1}^r \alpha_i(x)$ is coprime with $r$. The isomorphism class of the complex analytic space $\SM_{\op{Hod}}(r,\alpha,\xi)$ determines uniquely the isomorphism class of the punctured Riemann surface $(X,D)$.
\end{corollary}

\begin{proof}
We will proceed similarly to the proof of Corollary \ref{cor:torelliHiggs}. Let $Z\subset \SM_{\op{Hod}}(r,\alpha,\xi)$ be a closed analytic subset with the following two properties:

\begin{enumerate}
\item $Z$ is irreducible  and has complex dimension $(r^2-1)(g-1)+ \frac{n(r^2-r)}{2}$.
\item The restriction of the holomorphic tangent bundle $T\SM_{\op{Hod}}(r,\alpha,\xi)$ to the subspace $Z^{sm}\subset \SM_{\op{Hod}}(r,\alpha,\xi)$ has no nonzero holomorphic sections.
\end{enumerate}

By Corollary \ref{cor:TangentHodgeSections}, $i(\SM(r,\alpha,\xi))\subset \SM_{\op{Hod}}(r,\alpha,\xi)$ of the embedding \eqref{eq:embeddingHodge} has these properties. We will prove that this is the only possible choice for $Z$.

Every $\CC^*$ action on $\SM_{\op{Hod}}(r,\alpha,\xi)$ defines a holomorphic vector field. The second assumption on $Z$ says that any such holomorphic vector field vanishes on $Z^{sm}$. Therefore, the stabilizer of each point $Z^{sm}\subset \SM_{\op{Hod}}(r,\alpha,\xi)$ has nontrivial tangent space at $1\in \CC^*$, and hence the stabilizer must be the full group $\CC^*$.

Then $Z^{sm}$ belongs to the fixed point locus of the action \eqref{eq:actionHodge} in $\SM_{\op{Hod}}(r,\alpha,\xi)$, and thus, so does its closure in $\SM_{\op{Hod}}(r,\alpha,\xi)$, $Z$. Due to Proposition \ref{prop:SimpsonHodge}, and property (1), $Z=i(\SM(r,\alpha,\xi))$. In particular, we have $Z\cong \SM(r,\alpha,\xi)$.

Then, as in the proof of Corollary \ref{cor:torelliHiggs}, the isomorphism class of $\SM_{\op{Higgs}}(r,\alpha,\xi)$ determines the isomorphism class of $\SM(r,\alpha,\xi)$. Due to \cite[Theorem 3.2]{TorelliParabolic}, the latter determines the isomorphism class of the punctured Riemann surface $(X,D)$.
\end{proof}

The rank two, coprimality and concentrated weights conditions of the previous corollary are only necessary in order to apply the Torelli theorem in \cite{TorelliParabolic}. If the Theorem of \cite{TorelliParabolic} were extended to higher rank or generic weights, then Corollary \ref{cor:torelliHodge} would also hold for higher rank and generic weights respectively with the same proof given above.

\section{The parabolic Deligne-Hitchin moduli space}
\label{section:ParDH}

We can extend Deligne's construction \cite{De} of the Deligne-Hitchin moduli space for the group $\op{SL}(r,\CC)$ as described in \cite[\S4]{BGHL} to the parabolic scenario. Let $\alpha$ be a full flag generic system of weights over $D$ such that for every $x\in D$,
\begin{equation}
\label{eq:integerParabolicWeights}
\beta(x):=\sum_{i=1}^r \alpha_i(x) \in \ZZ
\end{equation}

Let $X_\RR$ be the $C^\infty$ real manifold of dimension two underlying $X$. Fix a point $x_0\in X_\RR \backslash D$. For every $x\in D$, let $\gamma_x\in \pi_1(X_\RR\backslash D,x_0)$ be the class of a positively oriented simple loop around $x$. Let $\SM_{\op{rep}}(X_\RR, r,\alpha)$ be the subvariety of $\Hom(\pi_1(X_\RR \backslash D, x_0), \op{SL}(r,\CC)) \gitq \op{SL}(r,\CC)$ corresponding to classes of irreducible representations $\rho : \pi_1(X_\RR\backslash D, x_0) \longrightarrow \op{SL}(r,\CC)$ such that for each $x\in D$, $\rho(\gamma_x)$ has eigenvalues $\{e^{-2\pi i\alpha_i(x)}\}$. The group $\op{SL}(r,\CC)$ acts on $\Hom(\pi_1(X_\RR \backslash D, x_0), \op{SL}(r,\CC))$ through the adjoint action of $\op{SL}(r,\CC)$ on itself. Since the eigenvalues of $\rho(\gamma_x)$ are preserved by conjugation, the quotient is well defined. On the other hand, the determinant of $\rho(\gamma_x)$ is the product of its eigenvalues, so

$$\det(\rho(\gamma_x))=\prod_{i=1}^r e^{-2\pi i \alpha_i(x)}= e^{-2\pi i \sum_{i=1}^r \alpha_i(x)}=1$$

The fundamental groups for different base points are identified up to an inner automorphism and the different choices of the loops $\gamma_x$ are identified through an outer isomorphism. Thus, the isomorphism class of the space $\SM_{\op{rep}}(X_\RR, r, \alpha)$ is independent of the choice of $x_0$ and the loops $\gamma_x$, so we can omit any reference to both of them.

A Riemann-Hilbert-like correspondence can be defined between the moduli space of stable parabolic connections of parabolic degree 0 and the moduli space of stable filtered local systems \cite{SimpsonNonCompact} of degree 0 for the group $\GL(n,\CC)$. Simpson proved that this correspondence is a functor compatible with determinants \cite[Lemma 3.2]{SimpsonNonCompact}. This implies that there exist a line bundle $\xi$ over $X$ such that there exist a biholomorphic isomorphism
\begin{equation}
\label{eq:riemann-hilbert}
\SM_{\op{rep}}(X_\RR, r, \alpha) \stackrel{\sim}{\longrightarrow} \SM_{\op{conn}}(r,\alpha,\xi)
\end{equation}

Let $\SM_{\op{conn}}(r,\alpha)$ denote the moduli space of stable parabolic connections of rank $r$, parabolic weights $\alpha$ and parabolic degree 0 over $X$ and let $\SM_{\op{rep}}(X_\RR,\alpha,\GL(r,\CC))$ denote the space of classes of irreducible representations of the fundamental group of $X_\RR \backslash D$ in $\GL(r,\CC)$ with weights $\alpha$. Then the following diagram commutes

\begin{eqnarray}
\label{eq:RHdet}
\xymatrixcolsep{3pc}
\xymatrix{
\SM_{\op{conn}}(r,\alpha) \ar[d]_{\det} & \SM_{\op{rep}}(X_\RR, \alpha, \GL(r,\CC)) \ar[l]_-{\op{RH}} \ar[d]_{\det}\\
\SM_{\op{conn}}(1,\beta) & \SM_{\op{rep}}(X_\RR, \beta, \CC^*) \ar[l]_-{\op{RH}}
}
\end{eqnarray}

$\SM_{\op{rep}}(X_\RR,r,\alpha)\subset \SM_{\op{rep}}(X_\RR,\alpha,\GL(r,\CC))$ is the fiber of the determinant over the equivalence class of the constant representation $[\rho]\in\SM_{\op{rep}}(X_\RR,\beta,\CC^*)$ given by $\rho(\gamma)=1$ for all $\gamma\in\pi_1(X_\RR \backslash D,x_0)$.

By construction, the image of the constant representation by the Riemann-Hilbert correspondence for weight $0$ is simply the trivial line bundle over $X$, i.e., $\SO_X$, together with the trivial connection. Thus, the image of weight $\beta$ corresponds to the line bundle
$$
\xi:=\SO_X \left (\sum_{x\in D} \beta(x) x \right)
$$
with the trivial parabolic structure of weight $\beta$ and the connection $\nabla_{\xi,\beta}$.

The Riemann-Hilbert correspondence is bijective. Thus, by the commutativity of diagram \eqref{eq:RHdet}, it sends the fiber of the determinant over $[\rho]$ to the fiber of the determinant in $\SM_{\op{conn}}(r,\alpha)$ over $(\xi,\xi_\bullet,\nabla_{\xi,\beta})$. By definition \eqref{def:parabolicLambdaConnection}, this fiber is precisely $\SM_{\op{conn}}(r,\alpha,\xi)$. Therefore, the restriction of the Riemann-Hilbert correspondence to $\SM_{\op{rep}}(X_\RR,r,\alpha)$ induces the biholomorphic isomorphism \eqref{eq:riemann-hilbert}.

The isomorphism sends each representation $\rho : \pi_1(X_\RR\backslash D, x_0) \longrightarrow \op{SL}(r,\CC)$ to an associated parabolic vector bundle $(E_X(\rho),E_{X,\bullet}(\rho))$ over $X$ with weight system $\alpha$, endowed with a parabolic connection $\nabla_X(\rho)$.

Composing the isomorphism \eqref{eq:riemann-hilbert} with the action of $\CC^*$ in the moduli space of parabolic $\lambda$-connections given by \eqref{eq:actionHodge} gives us an embedding of $\SM_{\op{rep}}(X_\RR,\alpha) \hookrightarrow \op{pr}_\lambda^{-1}(\lambda)$ for every $\lambda\in \CC^*$. This defines a holomorphic open embedding
\begin{equation}
\label{eq:repEmbedding}
\CC^* \times \SM_{\op{rep}}(X_\RR, r, \alpha) \hookrightarrow  \SM_{\op{Hod}}(X,r,\alpha,\xi)
\end{equation}
onto the open locus $\op{pr}_\lambda^{-1}(\CC^*) \subset \SM_{\op{Hod}}(X,r,\alpha,\xi)$.

Let $J_X$ denote the almost complex structure of the Riemann surface $X$. Then $-J_X$ is also an almost complex structure on $X_\RR$. The Riemann surface defined by $-J_X$ will be denoted by $\overline{X}$. Similarly, let $\overline{\xi}$ be the vector bundle obtained with the conjugate almost complex structure of $\xi$. As a topological vector bundle it is isomorphic to $\xi^{-1}$.

We can also consider the moduli space $\SM_{\op{Hod}}(\overline{X},r,-\alpha,\overline{\xi})$ of parabolic $\lambda$-connections on $\overline{X}$, etc. Now, we define the parabolic Deligne-Hitchin moduli space
$$\SM_{\op{DH}}(X,r,\alpha) := \SM_{\op{Hod}}(X,r,\alpha,\xi) \cup \SM_{\op{Hod}}(\overline{X},r,-\alpha,\overline{\xi})
$$
by gluing $\SM_{\op{Hod}}(X,r,\alpha,\xi)$ to $\SM_{\op{Hod}}(\overline{X},r,-\alpha,\overline{\xi})$ along the image of $\CC^*\times \SM_{\op{rep}}(X_\RR,r,\alpha) \cong \CC^* \times \SM_{\op{rep}}(\overline{X_{\RR}},r,-\alpha)$ for the map in \eqref{eq:repEmbedding}. More precisely, we identify, for each $\lambda \in \CC^*$ and each representation $\rho\in \SM_{re\op{rep}}(X_\RR,r,\alpha)$, the two points
\begin{multline*}
(\lambda,E_X(\rho),E_{X,\bullet}(\rho),\lambda \cdot \nabla_X(\rho))\in \SM_{\op{Hod}}(X,r,\alpha,\xi)\\
 \text{and } (\lambda^{-1},E_{\overline{X}}(\rho),E_{\overline{X},\bullet}(\rho),\lambda^{-1} \cdot \nabla_{\overline{X}}(\rho))\in \SM_{\op{Hod}}(\overline{X},r,-\alpha,\overline{\xi})
\end{multline*}
The forgetful map $\op{pr}_\lambda$ in \eqref{eq:forgetfulHodge} extends to a natural holomorphic morphism
\begin{equation}
\label{eq:forgetfulDH}
\op{pr} :\SM_{\op{DH}}(X,r,\alpha) \longrightarrow \CC\PP^1
\end{equation}
whose fiber over $\lambda\in \CC\PP^1$ is canonically biholomorphic to

\begin{itemize}
\item the moduli space $\SM_{\op{Higgs}}(X,r,\alpha,\xi)$ of parabolic $\op{SL}(r,\CC)$ Higgs bundles on $X$ of weight system $\alpha$ and $\det(E) \cong \xi$ if $\lambda=0$
\item the moduli space $\SM_{\op{Higgs}}(\overline{X},r,-\alpha,\overline{\xi})$ of parabolic $\op{SL}(r,\CC)$ Higgs bundles on $\overline{X}$ of weight system $-\alpha$ and $\det(E) \cong \overline{\xi}$ if $\lambda=\infty$
\item the moduli space of parabolic $\lambda$-connections on $X$ of weight system $\alpha$ and $\det(E)\cong \xi$ for every fixed $\lambda \ne 0$ and $\lambda \ne \infty$. This fibers are also biholomorphic to the moduli space $\SM_{\op{rep}}(X_\RR, r, \alpha)$ of equivalence classes of representations $[\rho] \in \Hom(\pi_1(X_\RR \backslash D, x_0), \op{SL}(r,\CC)) \gitq \op{SL}(r,\CC)$ such that for some fixed loops $\gamma_x\in \pi_1(X_\RR \backslash D, x_0)$ around the points $x\in D$, $\rho(\gamma_x)$ has eigenvalues $\{e^{-2\pi i\alpha_i(x)}\}$.
\end{itemize}

Now we can prove the main result.

\begin{theorem}
\label{teo:TorelliParabolicDH}
Let $r=2$.  Let $X$ be a smooth complex projective curve of genus $g\ge 3$. Let $D$ be a finite set of $n\ge 1$ distinct points over $X$ and let $\alpha$ be a concentrated generic (in particular full flag) system of weights over $D$ such that for every $x\in D$,
$$\beta(x)=\sum_{i=1}^r \alpha_i(x)\in \ZZ$$
and $\sum_{x\in D}\beta(x)$ is coprime with $r$. The isomorphism class of the complex analytic space $\SM_{\op{DH}}(X,r,\alpha)$ determines uniquely the isomorphism class of the unordered pair of punctured Riemann surfaces $\{(X,D),(\overline{X},D)\}$.
\end{theorem}

\begin{proof}
As $\alpha$ is full flag and generic, $\SM_{\op{DH}}(X,r,\alpha)$ is smooth. Let
$$T\SM_{\op{DH}}(X,r,\alpha) \longrightarrow \SM_{\op{DH}}(X,r,\alpha)
$$
be its holomorphic tangent bundle. Since $\SM_{\op{Hod}}(X,r,\alpha,\xi)$ is open in $\SM_{\op{DH}}(X,r,\alpha)$, Corollary \ref{cor:TangentHodgeSections} implies that the restriction of $T\SM_{\op{DH}}(X,r,\alpha)$ to
$$i(\SM(X,r,\alpha,\xi)) \subset \SM_{\op{Hod}}(X,r,\alpha,\xi) \subset \SM_{\op{DH}}(X,r,\alpha)
$$
does not admit any nonzero holomorphic section. The same argument applies if we replace $X$ by $\overline{X}$. Since $\SM_{\op{Hod}}(\overline{X},r,-\alpha,\overline{\xi})$ is also open in $\SM_{\op{DH}}(X,r,\alpha)$, the restriction of $T\SM_{\op{DH}}(X,r,\alpha)$ to
$$i(\SM(\overline{X},r,-\alpha,\overline{\xi})) \subset \SM_{\op{Hod}}(\overline{X},r,-\alpha,\overline{\xi}) \subset \SM_{\op{DH}}(X,r,\alpha)
$$
does not admit any nonzero holomorphic section either. We will extend the $\CC^*$ action on $\SM_{\op{Hod}}(X,r,\alpha,\xi)$ in \eqref{eq:actionHodge} to $\SM_{\op{DH}}(X,r,\alpha)$. We consider the corresponding $\CC^*$ action on $\SM_{\op{Hod}}(\overline{X},r,\alpha,\xi)$. The action of any $t\in \CC^*$ on the open subset $\CC^* \times \SM_{\op{rep}}(X_\RR, r, \alpha) \longrightarrow \SM_{\op{Hod}}(X,r,\alpha,\xi)$ in \eqref{eq:repEmbedding} coincides with the action of $1/t$ on $\CC^* \times \SM_{\op{rep}}(X_\RR, r, \alpha) \longrightarrow \SM_{\op{Hod}}(\overline{X},r,-\alpha,\overline{\xi})$. Therefore, we get an action of $\CC^*$ on $\SM_{\op{DH}}(X,r,\alpha)$.

Due to Proposition \ref{prop:SimpsonHodge}, each irreducible component of the fixed point locus of this $\CC^*$ action on $\SM_{\op{DH}}(X,r,\alpha)$ has dimension less or equal to $(r^2-1)(g-1)+ \frac{n(r^2-r)}{2}$, with equality only for $i(\SM(X,r,\alpha,\xi))$ and for $i(\SM(\overline{X},r,-\alpha,\overline{\xi}))$.

In a similar way of the proof of Corollary \ref{cor:torelliHodge}, these observations imply that $\SM_{\op{DH}}(X,r,\alpha)$ determines the isomorphism class of the unordered pair of moduli spaces $\{\SM(X,r,\alpha,\xi),\SM(\overline{X},r,-\alpha,\overline{\xi})\}$. Therefore, using \cite[Theorem 3.2]{TorelliParabolic} the statement of the theorem follows.
\end{proof}

The rank two condition of the previous theorem  is only necessary in order to apply the Torelli theorem in \cite{TorelliParabolic}. If the Theorem of \cite{TorelliParabolic} were extended to higher rank, then Theorem \ref{teo:TorelliParabolicDH} would also hold for higher rank with the same proof given above.

\bibliographystyle{alpha}
\bibliography{torelliDeligneBib}

\end{document}